\documentclass{amsart}
\usepackage{amssymb,amsmath,amsthm}
\usepackage{enumitem}
\usepackage{microtype}
\usepackage{amsfonts}
\usepackage{mathtools}

\usepackage[]{hyperref}

\usepackage{cite}

\numberwithin{equation}{section}

\newcommand{\R}{\mathbb {R}}

\newtheorem{theorem}{Theorem}[section]
\newtheorem{prop}[theorem]{Proposition}
\newtheorem{lemma}[theorem]{Lemma}

\newtheorem{corollary}[theorem]{Corollary}

\theoremstyle{definition}
\newtheorem{definition}{Definition}[section]
\newtheorem{remark}[theorem]{Remark}

\newcommand{\beq}{\begin{equation}}
\newcommand{\eeq}{\end{equation}}

\newcommand{\beqq}{\begin{equation*}}
\newcommand{\eeqq}{\end{equation*}}

\theoremstyle{remark}

\newtheorem{example}{Example}[section]

\makeatletter
\newcommand{\Extend}[5]{\ext@arrow0099{\arrowfill@#1#2#3}{#4}{#5}}
\makeatother

\newcommand{\ft}[1]{\mathcal F{#1}}
\newcommand{\ift}[1]{\mathcal F^{-1}{#1}}
\newcommand{\ftp}[1]{\widetilde{\mathcal F}{#1}}

\begin{document}
\hspace{1em}
\title[A L.A.P. for high-order Schr\"odinger operators in critical spaces]{A Limiting absorption principle for high-order Schr\"odinger operators in  critical spaces}
\author{Xiaoyan Su}
\address{Xiaoyan Su, Laboratory of Mathematics and Complex Systems (Ministry of Education),
 School of Mathematical Sciences,
 Beijing Normal University,
 Beijing, 100875, People's Republic of China.}
\curraddr{}
\email{suxiaoyan0427@qq.com}

\author{Chengbin Xu}
\address{Chengbin Xu,
 The Graduate School of China Academy of Engineering Physics, P. O. Box 2101,
 Beijing, 100088, People's Republic of China.}
\curraddr{}
\email{xcbsph@163.com}

\author{Guixiang Xu}
\address{Guixiang Xu, Laboratory of Mathematics and Complex Systems (Ministry of Education),
 School of Mathematical Sciences,
 Beijing Normal University,
 Beijing, 100875, People's Republic of China.}
\curraddr{}
\email{guixiang@bnu.edu.cn}

\author{Xiaoqing Yu}
\address{Xiaoqing Yu, Laboratory of Mathematics and Complex Systems (Ministry of Education),
 School of Mathematical Sciences,
 Beijing Normal University,
 Beijing, 100875, People's Republic of China.}
\curraddr{}
\email{202021130031@mail.bnu.edu.cn}

\begin{abstract} In this paper, we prove a limiting absorption principle for high-order Schr\"odinger operators with a large class of potentials which generalize some results by A. Ionescu and W. Schlag. 
	Two key tools we use in this paper are the Stein--Tomas theorem in Lorentz spaces and a sharp trace lemma given by S. Agmon and L. H\"ormander%
.
\\[0.6em]
\\[0.6em]
\end{abstract}
\maketitle
\section{introduction}\label{section-intro}
 In this paper, we prove a limiting absorption principle for the high-order  Schr\"odinger operators
   \[H= (-\Delta)^{m} + V, \quad m\ge 1, \]
   with a large class of potentials $V$ belonging to certain Banach spaces, which generalize some  results by A. Ionescu and W. Schlag in \cite{I-S}.
   Limiting absorption principles play an important role in the spectral theory as they directly imply the absence of  singular spectrum.
  They are closely related to the spectral measures of the Schr\"odinger operators, and the distorted Fourier transforms, and the long-time behavior of the unitary group $e^{itH}$.

Due to their wide applications, many different limiting absorption principles have been proven over the decades. S. Agmon \cite{Ag} established a limiting absorption principle for elliptic operators
with short-range potentials in weighted $L^2$ spaces using the trace lemma.
%
  T. Ikebe and Y. Saito in \cite{Ike-Saito} proved similar limiting absorption principles for second-order differential operators with magnetic term using \emph{a priori} estimates. By introducing a critical space for the trace lemma, L. H\"ormander established a limiting absorption principle for simply characteristic differential operators with generalized short-range potentials in \cite{Hom}.  Some quantitative limiting absorption principles were investigated for Schr\"odinger operators on asymptotically conic manifolds with short-range potentials by I. Rodnianski and T. Tao in \cite{Rod-Tao}.

By the connection between resolvent operators and restriction operators, M. Goldberg and W. Schlag established  a limiting absorption principle  in $L^p$ spaces instead of weighted $L^2$ spaces in \cite{GS}  by the  Stein--Tomas restriction theorem. This new type of
  limiting absorption principle allowed them to handle the Schr\"odinger operators with singular potentials.  However, the class of potentials investigated  is quite restricted in \cite{GS} since the results  rely on the unique continuation results established by A. Ionescu and D. Jerison in \cite{IJ}. In order to avoid using the unique continuation results, A. Ionescu and W. Schlag in \cite{I-S} proved  the asymptotic completeness of Schr\"odinger operators for a class of potentials that include some singular potentials, a ``global" Kato class, as well as some first-order differential operators.  Their main innovation was a new limiting absorption principle in some critical spaces which contains both the  cases in \cite{A-H} and \cite{GS}. This general limiting absorption principle extends the Agmon--Kato--Kuroda theorem  to a large class of perturbations.  See also \cite{Huang-Yao-Zheng} for more related results.

{The goal of this paper is to establish a limiting  absorption principle for high-order Schr\"odinger operators
 in some critical spaces inspired by the work of A. Ionescu and  W. Schlag in  \cite{I-S}.
  In particular, we also extend some results in \cite{I-S} by using the Stein--Tomas restriction theorem in Lorentz spaces instead of Lebesgue spaces.}
    In order to establish our limiting absorption principle, we only need to consider the boundary operators for the resolvents using the  maximum principle for analytic functions.
The boundary operator is closely related to the Fourier restriction operators: the imaginary part of the boundary operator is the restriction operator. This relationship allows us to prove the required estimates in the critical spaces by using
  the Stein--Tomas restriction theorem and the sharp trace lemma.

The critical spaces we use in our paper allow us to include more general potentials, however we are unable to exclude the embedded eigenvalues of the corresponding operators. It therefore seems likely that other new techniques are required to attain this goal.
We remark that this has been achieved for some special kinds of Schr\"odinger operators; for instance, T. Kato  proved the non-existence of positive eigenvalues in \cite{Kato},  which implies the absence of embedded eigenvalues for the Schr\"odinger operators with short-range potentials.
In \cite{IJ}, A. Ionescu and D. Jerison  used unique continuation properties to exclude embedded eigenvalues for the Schr\"odinger operators with certain singular potentials.
H. Koch and D. Tataru  excluded the positive eigenvalues for Schr\"odinger operators with more general potentials in \cite{KT}.  We leave this topic for further discussion.

\subsection{Notations}\label{section-prelim}
In this subsection, we explain some notations used in this paper. Let $d\geq 2$ be the dimension and $m$ be any integer larger than zero.

 We set $P_m(\xi)=|\xi|^{2m}$  so that $P_m(D)=(-\Delta)^m$ is the $m$th-order Laplacian (which is a differential operator of order $2m$).
Let  $R_0^m$ and $R^m$ denote the resolvents for $P_m (D)$ and $P_m (D)+V$ respectively.


For $r>0$, we write $\mathbb S_r ^{d-1}=\{\xi: P_m(\xi)=r^{2m}\}$, and write $d\sigma_r $ for the surface measure on $\mathbb S_r ^{d-1}$. In order to simplify notation, we will often reason with $\lambda>0$ instead of $r$, where the two are related by
\begin{align}
    r(\lambda)=\lambda^{\frac{1}{2m}}.\label{notation-r-lambda}
\end{align}
Let $\mathcal{B}(X, Y)$ be the set of bounded linear operators $T:X\to Y$ between Banach spaces $X$ and $Y$.

For any $\alpha \in {\mathbb C }$, let $S_{\alpha}: \mathcal{S}'(\mathbb R^d)\rightarrow \mathcal{S}'(\mathbb R^d)$ denote the Fourier multiplier operators given by
$$S_\alpha u(x)=\{(1+|\xi|^2)^{\alpha/2}\hat u(\xi)\}^{\vee}(x).
$$
For $1<p<\infty$ and $\alpha\in \mathbb R$, we define the following spaces
\begin{align*}
	W_\alpha ^{p}(\mathbb R^d)&=\{u\in\mathcal{S}'(\mathbb R^d): S_{\alpha}u \in L^{p}(\mathbb R^d)\},\\
	W_{\alpha, \rm{loc}} ^{p}(\mathbb R^d)&=\{u\in\mathcal{S}'(\mathbb R^d):S_{\alpha}u \in L_{\rm{loc}}^{p}(\mathbb R^d)\},\  \text{and}   \\
	W_\alpha ^{p,q }(\mathbb R^d)&=\{u\in\mathcal{S}'(\mathbb R^d):S_{\alpha}u \in L^{p, q}(\mathbb R^d)\},
\end{align*}
where $L^{p, q}(\mathbb R^d)$ are the Lorentz spaces (see \cite{Stein-Weiss}).

Let $p_d= (2d+2)/(d+3)$ and $p_d'=(2d+2)/(d-1)$ denote the Stein--Tomas restriction exponents.
The Stein--Tomas restriction theorem in Lorentz spaces, which states that
$	\|\ft{ f}\|_{L^2{(\mathbb S_{r}^{d-1}})} \leq C_r \|f\|_{L^{p_d, 2}(\mathbb R^d)}$, is  due to Bak and Seeger \cite{BS}. Here,  $\ft{f}(
\xi)=\int_{\mathbb R^n} f(x) e^{i\xi \cdot x} dx$ is the Fourier transform.

 Another important result related to the restriction operator is given by the trace lemma (see \cite{Ag}).
The trace lemma asserts that if a function belongs to the weighted  space $L^2_s(\mathbb R^d)$ with $s>1/2$, then its Fourier transform restricted to a sphere belongs to $L^2(d\sigma)$ (see \cite{Ag}). It is not true for $s=\frac{1}{2}$. However a critical space $B$ that was introduced by S. Agmon and L. H\"ormander in \cite{A-H}  is a good replacement for $L^{2}_{ 1/2}(\mathbb R^d)$ in the sense that
$$\mathcal{F}:B\rightarrow L^2(\mathbb {S}^{d-1})\ \   \text{and}  \ \  \mathcal{F}^{-1}: L^2(\mathbb {S}^{d-1})\rightarrow B^*,$$ as bounded operators.

The space $B$ and its dual $B^*$ is defined as follows. Let $D_j=\{x\in\mathbb R^d: 2^{j-1}\leq |x|\leq 2^j\}$, for $\ j\geq1$, and $D_0=\{x\in\mathbb R^d: |x|\le 1\}$.  Then
\begin{align*}B&=\bigg\{f:\mathbb R^d\rightarrow\mathbb C\ \bigg|\ \|f\|_{B}\coloneqq\sum_{j=0}^{\infty}2^{j/2}\|f\|_{L^2(D_j)}<\infty\bigg\},\\
\shortintertext{and}
B^*&=\bigg\{f:\mathbb R^d\rightarrow\mathbb C\ \bigg|\  \|f\|_{B^*}\coloneqq\sup_{j\geq0}2^{-j/2}\|f\|_{L^2(D_j)}<\infty\bigg\}.
\end{align*}
From Theorem 14.1.2 in \cite{Hom}, we have $B \hookrightarrow L^1 (\mathbb R; L^2(\mathbb R^{d-1})) $,
\begin{align}\label{embedded B}
	\int_{-\infty}^\infty \|f(\cdot , x_d)\|_{L^2(\mathbb R^{d-1})} dx_d \leq \sqrt{2} \|f\|_B.
\end{align}
It follows by duality that $L^\infty(\mathbb R; L^2(\mathbb R^{d-1})) \hookrightarrow B^*$, i.e.
\begin{align}\label{dual embedding}
	\|f\|_{B^*} \leq \sqrt{2} \sup_{x_d \in \mathbb R} \|f(x',x_d)\|_{L^2_{x'}}.
\end{align}

Let $S_\alpha(B)$ and $S_\alpha(B^*)$ denote the image of $B$ and $B^*$ under $S_\alpha$ respectively.
As we mentioned before, we will prove a limiting absorption principle for $P_m(D)$ in some critical Banach spaces in $\mathbb R^d$ with $d\geq 2$. These spaces are
\begin{align*}
    X \coloneqq W_{-\theta_{m,d}}^{p_d, 2}(\mathbb R^d)+S_m(B)\ \  \text{and}\ \  X^*& \coloneqq W_{\theta_{m,d}}^{p_d', 2}(\mathbb R^d)\cap S_{-m}(B^*),
\end{align*}
where $\theta_{m,d}=m-\frac{d}{d+1}$.
Using  the Sobolev embedding theorem and the fact that
$B\hookrightarrow L^2(\mathbb R^d) \hookrightarrow B^* \hookrightarrow L^2_{loc}(\mathbb R^d)$, it is easy to get the following embeddings,
\begin{align}\label{embedding result}
	X\hookrightarrow W_{-m}^{2},\quad W_{m}^2\hookrightarrow X^*\hookrightarrow W_{m,\rm{loc}}^2.
\end{align}
By  sharp trace theorem  and Stein-Tomas theorem, we have
$$\mathcal{F}:X\rightarrow L^2(\mathbb {S}^{d-1})\ \  \text{and}\ \  \mathcal{F}^{-1}: L^2(\mathbb {S}^{d-1})\rightarrow X^*$$
as bounded operators.

\subsection{Main results} The first main result in this paper is the following limiting absorption principle for the high-order Laplacians:
 \begin{theorem}\label{Main theorem -1}
Assume that $\delta \in (0, 1]$. Then
\begin{align}\label{LAP}
	\sup_{\lambda\in [\delta, \delta^{-1}], \ \varepsilon \in [0,1]}\|R^m_0(\lambda\pm i\varepsilon )\|_{X\to X^{*}}
	\leq C_{\delta,d}<\infty.
\end{align}

\end{theorem}
In order to prove Theorem \ref{Main theorem -1},
we write the resolvent as an integral operator with a Schwartz kernel. By studying this Schwartz kernel, one is able to derive all the required estimates.
 Essentially, modulo a good operator, the resolvent operator is a  restriction operator on a small spherical  shell around a sphere, which leads to estimates similar to those for the restriction operator.
 Using the exact form of the kernel, we are able to reduce our operator to the restriction operator with a weight which is almost constant in the shell.

We emphasize that our method  does not use the fundamental solution of $P_m(D)$. Therefore, we do not impose any upper bounds on $m$.

The second main result is the following properties for the perturbed operators:

\begin{theorem} \label{Main theorem 2}
	 Assume that $V$ is an admissible perturbation (see Definition \ref{admissible definition} below).
	 \begin{enumerate}
	     \item 	 Then
the operator $H_m=P_m(D)+V$ defines a self-adjoint operator on
  $$D(H_m)\coloneqq\{u\in W_m^2(\mathbb R^d): H_m u\in L^2(\mathbb R^d)\}.$$
  In addition, $D(H_m)$ is dense in $L^2(\mathbb R^d)$, and $H_m$ is bounded from below on $D(H_m)$.
\item Let $I \subset (\mathbb R\setminus \{0\}) \setminus \mathcal E$ be compact. Then
	\begin{align}\label{LAP-V}
	\sup_{\lambda\in I, \varepsilon \in (0,1]}\|R^m(\lambda\pm i\varepsilon )\|_{X\to X^{*}}
	\leq C(V, I)<\infty,
\end{align}
where $\mathcal E$ is the set of nonzero discrete eigenvalues of $H_m$. This implies that the spectrum of the operator $H_m$ is purely absolutely continuous on $I$. In particular, $\Omega^{\pm}(H_m, P_m(D))$ exist and are complete, and $\sigma_{sc}(H_m)=\emptyset, \sigma_{ac}(H_m)=[0, \infty)$.
	 \end{enumerate}

 \end{theorem}

To get this result, we use the important Corollary \ref{Weighted uniform bound} (see Section 2), which is deduced from Theorem \ref{Main theorem -1}. This Corollary and a Rellich theorem established in \cite{A-H} gives the regularity for the eigenfunctions.

 The rest of this paper is organized as follows. In Section 2, we prove a limiting absorption principle for high-order Laplacians. As a corollary, we deduce a weighted resolvent estimate. In Section 3, we define admissible potentials and prove that the same limiting absorption principle is also valid for admissible perturbations.

\subsection*{Acknowledgements}
G. Xu  was supported by National Key Research and Development Program of China (No. 2020YFA0712900) and by NSFC ( No. 11831004).

\section{The Limiting Absorption Principle for $P_m(D)$  in $\mathcal B(X, X^*)$}\label{section-LAP}

In this section, we show that for every $\lambda>0$, as $\varepsilon$ goes to $0^{+}$, the weak-$*$ limit of $R_0^m(\lambda\pm i \varepsilon )$ exists as an operator $X\to X^*$.  In order to do that,  we first prove that the limit exists as an operator  $\mathcal S\to \mathcal S'$ using Privalov's theorem. Then we verify that it is bounded from  $X$ to $X^*$, and that the norm is locally uniform in $\lambda$.

For any $z=\lambda\pm i\varepsilon \in \mathbb C \setminus \mathbb R$ and $f,g\in \mathcal{S}(\mathbb R^d)$, we have
\begin{align*}
  \langle R^m_0(\lambda\pm i\varepsilon )f,g\rangle&=\frac{1}{(2\pi)^d}\int_{\mathbb R^d}\frac{\hat{f}(\xi)\overline{\hat{g}}(\xi)}{P_m(\xi)-(\lambda \pm i\varepsilon)}d\xi\\
  &=\frac{1}{(2\pi)^d}\frac{1}{2m}\int_{0}^{\infty}
  \frac{t^{(d-2m)/2m}}{t-(\lambda \pm i\varepsilon)}\left(\int_{\mathbb S^{d-1}}\hat{f} (t^{1/2m}w)\overline{\hat{g}}(t^{1/2m}w)dw\right)dt.
\end{align*}
If we denote \[\gamma(t): =t^{(d-2m)/2m}\int_{\mathbb S^{d-1}}\hat{f}(t^{1/2m}w)\overline{\hat{g}}(t^{1/2m}w)dw.\]
 It is easy to check that $\gamma(t)$ is locally H\"older continuous in $t \in (0, \infty)$. Hence, by Privalov's theorem, the boundary operator $R^m_0(\lambda\pm i0 )  =\lim_{\varepsilon\to 0}R^m_0(\lambda\pm i\varepsilon )$ exists as an $\mathcal B(\mathcal S(\mathbb R^d), \mathcal S'(\mathbb R^d))$-valued function for every $\lambda >0$ and are locally H\"older continuous. Furthermore,  for any $f, g \in \mathcal S (\mathbb R^d)$,  we can write the boundary operator as
\begin{align}
\langle R^m_0(\lambda &\pm i0)f,g\rangle  =\langle (P_m(\xi)-(\lambda \pm i0))^{-1}\hat f,\hat{g}\rangle\nonumber\\
&=\pm \frac{1}{(2\pi)^d}\frac{i\pi}{2m}\int_{P_m(\xi)=\lambda}\hat{f}(\xi)
  \overline{\hat{g}}(\xi)d\sigma
  +\frac{1}{(2\pi)^d}{\operatorname{p.v.}}\int_{\mathbb R^d}\frac{\hat{f}(\xi)\overline{\hat{g}}(\xi)}{P_m(\xi)-\lambda}d\xi\label{Cauchy integral formula}.
 \end{align}
 Here, $\rm {p.v.}$ means that the integral is to be understood in the principle value sense.

Using the Fourier multiplier theorem and the first embedding in \eqref{embedding result}, it is easy to check for any   $ z \in \mathbb C \setminus [0, \infty)$ that we have the estimate
\begin{align*}
	|\langle R_{0} ^m (z )f, g \rangle |\leq C_{z}\|f\|_{W_{m}^{2}}\|g\|_{W_{m}^{2}}\leq C_{z}\|f\|_{X}\|g\|_{X}.
\end{align*}
That is, $R^m_0(z)$ is a $\mathcal B(X, X^*)$-valued analytic operator in $\mathbb C \setminus [0, \infty)$. Using the maximum principle for  analytic functions,
  we will have for all $\lambda \in [\delta, \delta^{-1}]\  (\delta \in (0,1])$,
  \begin{align*}
  \sup_{\lambda \in [\delta, \delta^{-1}],\varepsilon\in [0,1] }|\langle R_{0} ^m (\lambda \pm i\varepsilon )f, g \rangle |&\leq C_{\delta,d} \|f\|_{X}\|g\|_{X},\\
\shortintertext{provided we can prove that }
  	\sup_{\lambda \in [\delta, \delta^{-1}] }|\langle R_{0} ^m (\lambda \pm i0)f, g \rangle |&\leq C_{\delta,d} \|f\|_{X}\|g\|_{X}.
  \end{align*}

Thus, in the following, we give the uniform bound for the boundary operators which implies Theorem \ref{Main theorem -1} using the maximum principle for analytic functions.
\begin{prop}\label{main proposition}
Assume that $\delta \in (0, 1]$. Then
\begin{align}
	\sup_{\lambda\in [\delta, \delta^{-1}]}\|R^m_0(\lambda\pm i0)\|_{X\to X^{*}}
	\leq C_{\delta,d} <\infty.
\end{align}
\end{prop}

\begin{proof}


 Let $\chi_\lambda:\mathbb R^d\rightarrow [0,1]$ denote a smooth function supported in $\{\xi: P_m(\xi)\in [{\lambda}/2, \frac{3}{2}\lambda] \}$ and equal to $1$ in  $\{\xi:P_m(\xi)\in [\frac{3}{4}{\lambda}, \frac{5}{4}\lambda] \}$. We write $\chi_\lambda(D)$ and $(1-\chi_\lambda)(D)$ for the corresponding Fourier multiplier operators. This allows us to split
 \[R^m_0(\lambda\pm i0)=\chi_\lambda(D) R^m_0(\lambda\pm i0) + (1-\chi_\lambda)(D)R^m_0(\lambda\pm i0).\]
 For the operator $(1-\chi_\lambda)(D)R^m_0(\lambda\pm i0)$, notice that $(1-\chi_\lambda(\xi))\frac{(1+|\xi|^{2})^m}{|\xi|^{2m}-\lambda} \in C^\infty(\mathbb R^d)$ with $L^\infty$ norm bounded uniformly in $\lambda$. Therefore, this part of the operator satisfies the stronger estimate
 \begin{align}\label{weighted nonsingular part}
   \|(1-\chi_\lambda)(D)R^m_0(\lambda\pm i0)\|_{W_{-m}^2\rightarrow W_{m}^2}\leq C_{\delta,d}.
 \end{align}
So we only need to prove that
\begin{align}\label{singular part}
\|\chi_\lambda(D)R^m_0(\lambda\pm i0)\|_{X\rightarrow X^*}\leq C_{\delta,d},
\end{align}
which by the definitions of $X$ and $X^*$ is equivalent to the following four inequalities:
\begin{alignat*}{4}
  &&\|S_m\chi_\lambda(D)R^m_0(\lambda\pm i0)S_m&\|_{B\rightarrow B^*} &\leq C_{\delta,d},\\
  &&\|S_{\theta_{m,d}}\chi_\lambda(D)R^m_0(\lambda\pm i0)S_{\theta_{m,d}}&\|_{ L^{p_d,  2}(\mathbb R^d)\rightarrow L^{p_d', 2}(\mathbb R^d)}&\leq C_{\delta,d},\\
  &&\|S_{m}\chi_\lambda(D)R^m_0(\lambda\pm i0)S_{\theta_{m,d}}&\|_{ L^{p_d,  2}(\mathbb R^d)\rightarrow B^*} &\leq C_{\delta,d} ,\\
  &&\|S_{\theta_{m,d}}\chi_\lambda(D)R^m_0(\lambda\pm i0)S_{m}&\|_{B\rightarrow L^{p_d', 2}(\mathbb R^d)} &\leq C_{\delta,d}.
\end{alignat*}
By duality, the ($B\to L^{p_d',  2}$) bound follows from $(L^{p_d,  2} \to B^*)$ bound.

Notice that, for any fixed $\lambda>0 $ and $\alpha>0$,  $S_\alpha(\xi) \chi_{\lambda} (\xi)\in C_c^\infty(\mathbb R^d)$, and for any $\beta \geq 0, $
\begin{align*}
	\|D^\beta [S_\alpha(\xi) \chi_{\lambda} (\xi)]\|_{\infty} \leq C_{\beta} (1+\lambda^2)^{\alpha/2}.
\end{align*}
Thus, by Corollary 2.6 in \cite{A-H}, Fourier multiplier operator $S_\alpha(D) \chi_{\lambda}(D)$ is bounded  on $B$. Also, a standard multiplier theorem asserts that it is bounded on $L^p$ for $1\leq p\leq \infty$. Thus using the interpolation theorem on Lorentz spaces%
, it is also bounded on  $L^{p_d, 2}(\mathbb R^d)$.
Furthermore, for any $\lambda \in [\delta, \delta^{-1}]$, we have the following uniform bounds,
\begin{align}
	\|S_\alpha \chi_{\lambda} f\|_{B}&\leq C_{\delta,d} \|f\|_{B}\label{bounded in $B$},
\\	\|S_\alpha \chi_{\lambda} f\|_{L^{p_d, 2}(\mathbb R^d)}&\leq C_{\delta,d} \|f\|_{L^{p_d, 2}(\mathbb R^d)} \label{Bounded in $L^{p_d, 2}$}.
\end{align}

Now, we define  $T^{\pm}_{\lambda}f\coloneqq [(P_m(\xi)-(\lambda\pm i0))^{-1}\chi_\lambda(\xi)\hat f(\xi)]^{\vee}(x)$.
Using  \eqref{bounded in $B$}, \eqref{Bounded in $L^{p_d, 2}$}, and by a  density argument, in order to prove \eqref{singular part},  we only need to prove the following inequality is valid for all $f\in \mathcal S(\mathbb R^d)$ ,
  \begin{align}\label{NL-1}
  \|T^{\pm}_\lambda f\|_{L^{p_d',2}(\mathbb R^d)\cap B^*}\leq C\|f\|_{L^{p_d,2}(\mathbb R^d)+B}
  \end{align}
In the following, we only give the details for $T_\lambda^{+}$; $T^-_\lambda$ can be treated similarly.

  Let $A=\{\xi_1, \dots, \xi_N\}$ denote a $\varepsilon_\delta/100$-net on
 $\mathbb S_{r(\lambda)}$, where $\varepsilon_\delta$ is a small positive constant which depends on $\delta$. Then using a partition of unity, we can write $\chi_\lambda =\chi_\lambda^0+\chi_\lambda^1+\cdots+\chi_\lambda^N$, where $\chi_\lambda^0$ is supported in the set $\{\xi: ||\xi|-r(\lambda)|\geq \frac{\varepsilon_\delta}{10}\}$ and $\chi_\lambda^j$ is supported in the set $\{\xi: |\xi^j-\xi|\leq \frac{\varepsilon_\delta}{2}\}$.  We use $\chi_\lambda^j(D)$ to denote the corresponding Fourier multiplier operators.  Using the same arguments as \eqref{weighted nonsingular part}, we have \begin{align}
   \|\chi_\lambda^0(D)R^m_0(\lambda\pm i0)\|_{W_{-m}^2\rightarrow W_{m}^2}\leq C_{\delta,d}.
 \end{align}

 For each $j\in \{1, \cdots, N\}$, using the rotation invariance, we may assume that the support of  $\chi_\lambda^j$ is near the north pole. From now on, we use $\chi_\lambda$ to denote $\chi_\lambda^j$ for simplicity.
 Thus, by the implicit function theorem, we can write $\xi_d=\varphi_\lambda (\xi')$ near $\mathbb S^{d-1}_{r(\lambda)} $, with $P_m(\xi',\varphi_\lambda (\xi'))-\lambda=0$.
  Note that $\frac{\chi_\lambda (\xi)(\xi_d-\varphi_\lambda (\xi'))}{P_m(\xi)-\lambda}$ has a removable singularity at $\xi_d = \varphi_\lambda(\xi')$, with $\lim_{\xi_d\to \varphi_\lambda(\xi')} \frac{\chi_\lambda (\xi)(\xi_d-\varphi_\lambda (\xi'))}{P_m(\xi)-\lambda} = \frac{\chi_\lambda (\xi)}{\frac{\partial P_m}{\partial\xi_d} (\xi',\varphi_\lambda(\xi'))}.$ Hence, by Taylor expansion around $\xi_d=\varphi_\lambda(\xi')$,
      \begin{align*}
      &\chi_\lambda (\xi)(P_m(\xi)-\lambda-i0)^{-1}\\
      =& \frac{1}{\xi_d-\varphi_\lambda (\xi')-i0}{\frac{\chi_\lambda (\xi)(\xi_d-\varphi_\lambda (\xi'))}{P_m(\xi)-\lambda}}
      \\
      =&\frac{1}{\xi_d-\varphi_\lambda (\xi')-i0}\underbrace{\frac{\chi_\lambda (\xi',\varphi_\lambda (\xi'))}{\frac{\partial P_m}{ \partial \xi_d}(\xi', \varphi_\lambda (\xi') )} }_{\coloneqq Q_{\lambda}(\xi')}+G_\lambda(\xi),
      \end{align*}
       where  $ Q_{\lambda}(\xi')\in C_c^{\infty}(\mathbb R^{d-1})$, $ G_\lambda(\xi)\in C_c^{\infty}(\mathbb R^d) $ and both are uniformly bounded in $\lambda \in [\delta, \delta^{-1}]$.
    Recall that the inverse Fourier transform of $\frac{1}{\xi_d-\varphi_\lambda (\xi')-i0}$ is given by
  \begin{align}
  	\frac{1}{2\pi}\int_{\mathbb R} e^{ix_d \xi_d}\frac{d\xi_d}{\xi_d-\varphi_\lambda (\xi')-i0}
  	=i e^{ i x_d \varphi_\lambda (\xi')} H(x_d),
  \end{align}
  where $H(t)$ is the Heaviside function.
   We define the main kernel for $T_\lambda ^{+}$ as the inverse Fourier transform of $\frac{1}{\xi_d-\varphi_\lambda (\xi')-i0} Q_\lambda(\xi')$, i.e.,
   \begin{align}
 	K_\lambda^+ (x', x_d)&= \frac{1}{(2\pi)^{d-1}}i H(x_d)\int_{\mathbb  R^{d-1}} e^{ i(x_d  \varphi_\lambda (\xi')+x'\xi')} Q_{\lambda}(\xi')d\xi' \label{first-expr-for-K+}	.
	 	 \end{align}

We now write  $T_\lambda^+$ as follows:
 \begin{align*}
 	T_\lambda ^{+}f(x', x_d)&= \int_{\mathbb R^d} K_\lambda^+(x'-y',x_d-y_d) f(y', y_d)dy' dy_d+\int_{\mathbb R^d} \ift{ G_{\lambda}}(x-y)f(y)dy\\
 	&\hspace{-0.2em}\coloneqq T_\lambda^{+\rm {bad}}f(x)+T_\lambda^{+\rm {good}} f(x),
\end{align*}
where $\ift{G_{\lambda}}$  is the inverse Fourier transform of $G_\lambda(\xi)$.
      Since $ G_\lambda(\xi)\in C_c^\infty(\mathbb R^d)$,   $\ift{G_\lambda}  \in \mathcal{S}(\mathbb R^d) $. Thus, $T_\lambda^{+\rm {good}} f(x)$ satisfies the following estimate
    \begin{align*}
    	\Big\|\int_{\mathbb R^d}\ift{G_{\lambda}}(x-y)f(y)dy\Big\|_{L^{p_d',2}(\mathbb R^d)\cap B^*}\leq C\|f\|_{L^{p_d,2}(\mathbb R^d)+B}.
    \end{align*}
It remains to prove the  \eqref{NL-1} is valid for the main term $T_\lambda^{+\rm {bad}}f(x)$.

Firstly, we prove the $ (B \to B^*)$ bound.  Notice that for the kernel $K_{\lambda}^+(x)$, by \eqref{first-expr-for-K+} we have, writing $\ftp{f}(\xi',x_d) =\int_{\mathbb R^{d-1}} f(x',x_d) e^{-ix'\cdot\xi'} dx'$ for the partial Fourier transform in $x'$,
\begin{align}\label{Kernel property-1}
	   \sup_{x_d} \big\| \ftp{ K_\lambda^+}(\xi', x_d)\big \|_{L^\infty(\mathbb R^{d-1})}=  \sup_{x_d} \sup_{\xi'\in \mathbb R^{d-1}} \big| i e^{ i x_d \varphi_\lambda (\xi')} H(x_d)Q_{\lambda}(\xi')\big |<\infty,
\end{align}
By \eqref{embedded B}, \eqref{dual embedding}, Minkowski's inequality, Plancharel, and \eqref{Kernel property-1}, we have
\begin{align}\label{B to B^* bound}
	\|T_\lambda^{+\rm{bad}}f(x)\|_{B^{*}}&\leq C \sup_{x_d \in \mathbb R} \|T_\lambda^{+\rm{bad}}f(   \cdot, x_d)\|_{L^2_{x'}}\nonumber\\
	&\leq C\sup_{x_d \in \mathbb R}\int_{\mathbb R}\left\|\int_{\mathbb R^{d-1}}K_\lambda^+(x'-y', x_d-y_d)f(y', y_d)d{y'}\right\|_{L^2_{x'}}d{y_d}\nonumber\\
	&\leq C\sup_{x_d \in \mathbb R} \int_{\mathbb R}\left\| e^{ i (x_d-y_d)\varphi_\lambda (\xi')} H(x_d-y_d)Q_{\lambda}(\xi')\ftp{f}(\xi', y_d)\right\|_{L^2_{\xi'}}d{y_d}\nonumber\\
	&\leq C\int_{\mathbb R}\left\|\ftp{f}(\xi', y_d)\right\|_{L^2_{\xi'}}d{y_d} \leq  C\int_{\mathbb R}\left\|f(x', y_d)\right\|_{L^2_{x'}}d{y_d} \lesssim \|f\|_B.
\end{align}

    Secondly, to get the $(L^{p_d,2} \to B^*) $ bound, let us define $h_{x_d}(y', y_d)\coloneqq i H(x_d-y_d)f(y', y_d)$ for each fixed $x_d$.
    It is easy to check that $$\sup_{x_d}\|h_{x_d}\|_{L^{p_d,2}(\mathbb R^d)}\\
  \leq \|f\|_{L^{p_d,2}(\mathbb R^d)}.$$
Then,
  \begin{align}\label{From L^{p_d,2} to B^*}
 T_\lambda^{+\rm{bad}}&f(x', x_d)= \int_{\mathbb R}\int_{\mathbb R^{d-1}} K_\lambda^+(x'-y',x_d-y_d) f(y', y_d)dy' dy_d \nonumber\\
  	&=\int_{\mathbb R}\int_{\mathbb R^{d-1}} e^{i x' \xi'}\ftp K_\lambda^+(\xi',x_d-y_d) \ftp f(\xi', y_d)d\xi' dy_d\nonumber\\
  	&=  \int_{\mathbb R}\int_{\mathbb R^{d-1}} e^{i x' \xi'}i H(x_d-y_d)e^{i(x_d-y_d)\varphi_\lambda (\xi')}\ftp f(\xi', y_d)Q_{\lambda}(\xi') d {\xi'} dy_d\nonumber\\
  	&= \int_{\mathbb R^{d-1}} e^{i x' \xi'} e^{i x_d \varphi_\lambda (\xi')}\int_{\mathbb R} i H(x_d-y_d)e^{-iy_d \varphi_\lambda (\xi')}\ftp f(\xi', y_d)dy_d  Q_{\lambda}(\xi')d {\xi'}.
  \end{align}
That is,
  \begin{align*}
   T_\lambda^{+\rm{bad}}f(x', x_d)= &\int_{\mathbb R^{d-1}}e^{ix'\xi'}e^{i x_d \varphi_\lambda (\xi')}\ft{ h_{x_d}}(\xi', \varphi_\lambda (\xi'))Q_{\lambda}(\xi')d\xi'.
  \end{align*}
By \eqref{dual embedding}, the Plancharel theorem and the Stein--Tomas restriction theorem, we have
\begin{align*}
  \| T_\lambda^+ f(x', x_d)\|_{B^*}&\leq C \sup_{x_d}\left\|\int_{\mathbb R^{d-1}}e^{ix'\xi'}e^{i x_d \varphi_\lambda (\xi'))}Q_{\lambda}(\xi')\ft{h_{x_d}}(\xi', \varphi_\lambda (\xi'))d\xi'\right\|_{L^2_{x'}(\mathbb R^{d-1})} \\
  &\leq C\sup_{x_d}\|Q_{\lambda}(\xi')\ft{h_{x_d}}(\xi', \varphi_\lambda (\xi'))\|_{L^2_{\xi'}(\mathbb R^{d-1})}\\
  &\leq C_{\delta,d} \sup_{x_d}\|\ft{h_{x_d}}\|_{L^2(\mathbb S_{r(\lambda)}^{d-1})}\leq C_{\delta,d} \sup_{x_d}\|h_{x_d}\|_{L^{p_d,2}(\mathbb R^d)}\\
  &\leq C_{\delta,d} \|f\|_{L^{p_d,2}(\mathbb R^d)}.
\end{align*}

 Finally, to prove the $(L^{p_d, 2} \to L^{p_d', 2})$ bound,  we notice that  $|K_\lambda^+ (x)|\leq C_{\delta,d} (1+|x|)^{-\frac{d-1}2}$.  Using a dyadic decomposition and  interpolation, following the same argument to prove Theorem 6 in \cite{Gu} gives
\begin{align}
	\|T_\lambda^{+\rm{bad}}f(x)\|_{L^{p_d',2}(\mathbb R^d)}\leq C\|f\|_{L^{p_d,2}(\mathbb R^d)}.
\end{align}
This concludes the proof.
  \end{proof}

 \begin{remark}
The proof shows that this limiting absorption principle \eqref{LAP} is also valid for all homogeneous elliptic operators $P(D)$ satisfying
the non-degeneracy condition, that is,
$P(\xi)=\lambda$ for $\lambda>0$
 defines a compact hypersurface with nonvanishing Gaussian curvature.
 For example, one could use the finite type operators of (even) order $m\ge 2$, e.g. $P(D)=\xi_1^m+\cdots+\xi_d^m$.	
\end{remark}

 For $N\geq 0$, $\gamma \in (0,1]$, and $x\in \mathbb R^d$,  for $t
\in [0,\infty)$ we define the weight function
$${\mu_{N,\gamma}}(t)=\frac{(1+t^2)^N}{(1+\gamma t^2)^N},$$
and for $x\in \mathbb R^d$, and $\mu_{N, \gamma}(x)\coloneqq\mu_{N, \gamma}(|x|)$. This weight function was introduced in \cite{Hom} to prove the regularity of the eigenfunctions. 

 As an important corollary, we give the following weighted resolvent estimate.

\begin{corollary} \label{Weighted uniform bound}
Assume that $\delta \in (0, 1]$. Then
  \begin{align}
  \sup_{\lambda \in [\delta, \delta^{-1}]}\|\mu_{N,\gamma} R^m_0(\lambda\pm i0)\mu_{N,\gamma}^{-1}\|_{X\to X^{*}}
	\leq C_{\delta,d} <\infty.
  \end{align}
  \end{corollary}
We start with following multiplier lemmas  concerning the weight $\mu_{N,\gamma}$, which
  follows from Lemma 3.2 in \cite{I-S} and  restricted weak type interpolation theorem.
\begin{lemma}\label{A-1}
  Let $m(\xi)\in C^\infty(\R^d)$. If $m$ satisfies the differential bounds
  $$|\partial_\xi^\alpha m(\xi)|\leq C_\alpha(1+|\xi|^2)^{-|\alpha|/2},\ \ for\ \ any\ \  |\alpha|\geq0, $$
  then    if $p \in \{p_d,2, p_d'\}$ and $\mu\in\{\mu_{N,\gamma},\mu_{N,\gamma}^{-1}\}$, we have
  \begin{align*}
    \| \mu \varphi(D)\mu^{-1}\|_{L^{p,2}(\mathbb R^d)\rightarrow L^{p,2}(\mathbb R^d)}+\|\mu \varphi(D)\mu^{-1}\|_{B\rightarrow B}+\|\mu \varphi(D)\mu^{-1}\|_{B^*\rightarrow B^*}\leq C.
  \end{align*}
\end{lemma}

\begin{lemma}\label{bounedness of weighted multipliers}
  For $\alpha\in[-2m,2m]$, we have the following estimates
  \begin{align*}
    \|\mu S_{\alpha}\mu^{-1}S_{-\alpha}\|_{L^{p,2}(\mathbb R^d)\rightarrow L^{p,2}(\mathbb R^d)}+\|\mu S_{\alpha}\mu^{-1}S_{-\alpha}\|_{B\rightarrow B}+\|\mu S_{\alpha}\mu^{-1}S_{-\alpha}\|_{B^*\rightarrow B^*}\leq C,
  \end{align*}
   \begin{align*}
    \| S_{\alpha}\mu S_{-\alpha}\mu^{-1}\|_{L^{p,2}(\mathbb R^d)\rightarrow L^{p,2}(\mathbb R^d)}+\|S_{\alpha}\mu S_{-\alpha}\mu^{-1}\|_{B\rightarrow B}+\|S_{\alpha}\mu S_{-\alpha}\mu^{-1}\|_{B^*\rightarrow B^*}\leq C,
 \end{align*}
where $p \in \{p_d,2, p_d'\}$, $\mu\in\{\mu_{N,\gamma},\mu_{N,\gamma}^{-1}\}$.
\end{lemma}

\begin{proof}[Proof of Corollary \ref{Weighted uniform bound}]

As before, we split the operator into a singular part on the shell and a regular part away from the shell. For the regular part,  Lemma \ref{A-1} and Lemma \ref{bounedness of weighted multipliers} gives
  the following stronger inequality,
\begin{align}
	\|\mu_{N,\gamma} (1-\chi_\lambda(D))R^m_0(\lambda\pm i0)\mu_{N,\gamma}^{-1} \|_{W_{-m}^2 \to W_{m}^2}\leq C_{\delta,d}.
\end{align}
Using the embedding result \eqref{embedding result}, we have
\begin{align*}
	\|\mu_{N,\gamma} (1-\chi_\lambda(D))R^m_0(\lambda\pm i0)\mu_{N,\gamma}^{-1}\|_{X\to X^*} < C_{\delta,d}.
\end{align*}
Thus, it remains to prove that
\begin{align}
	\|\mu_{N,\gamma} \chi_\lambda(D) R^m_0(\lambda\pm i0)\mu_{N,\gamma}^{-1}\|_{X\to X^{*}}<C_{\delta,d}.
\end{align}

By  the proof of Theorem 1.2 in \cite{I-S} and Lemma \ref{bounedness of weighted multipliers}, this inequality is equivalent to the following boundness 
  \begin{align}\label{singular weighted estimate}
  	 \|\mu_{N,\gamma} \chi_\lambda(D) R^m_0(\lambda\pm i0)\mu_{N,\gamma}^{-1} f\|_{L^{p_d', 2}\cap B^*}\leq C_{N}\|f\|_{ L^{p_d,  2}(\mathbb R^d)+B}.
  \end{align}
  for all $f \in \mathcal S$.
Using the partition of unity and rotation invariance, we can assume that the support of $\chi_\lambda $ is contained in a neighborhood  of the north pole. In order to prove inequality \eqref{singular weighted estimate}, it  suffices to prove
\begin{equation}\label{one direction inequality}
  \|{\mu_{N,\gamma}}(x_d) R_0^m(\lambda \pm i0 )f\|_{L^{p_d'}(\mathbb R^d)\cap B^*}\leq C_{N,\delta}\|{\mu_{N,\gamma}}(x_d)f\|_{ L^{p_d,  2}(\mathbb R^d)+B},
\end{equation}
for $f\in\mathcal{S}(\mathbb R^d)$ and $\ft{f}$ is supported in $\{\xi:|\xi-\xi^+|\leq \varepsilon_0\}$,
where $\xi^+=(0,\dots, 0, r(\lambda))$.  In fact, let $\xi_1^+,...,\xi_d^+$ be a basis of $\mathbb R^d$ consisting of unit vectors in the ball $\{\xi:|\xi-\xi^+|\leq \varepsilon_0/2\}$. Clearly, for any $x\in \mathbb R^d$,
$$|x|\leq C(|x\cdot\xi_1^+|+\dots+|x\cdot\xi_d^+|).$$
Since
$$\mu_{N,\gamma}(x)\sim{\mu_{N,\gamma}}(x\cdot\xi_1^+)+\dots+\mu_{N,\gamma}(x\cdot\xi_d^+),$$
we can replace the weight $\mu_{N,\gamma}(x)$ with $\mu_{N,\gamma}(x\cdot\xi_j^+)$ in \eqref{singular weighted estimate}. In other words,  \eqref{singular weighted estimate} would follow from the following inequalities, $j=1,\dots, d$:
\begin{equation}
\label{key bound for weighted estimates}
	\|\mu_{N,\gamma}(|x\cdot \xi_j^{+}|) R_0^m(\lambda\pm i0 )f\|_{L^{p_d'}(\mathbb R^d)\cap B^*}\leq C_{N}\|\mu_{N,\gamma}(|x\cdot \xi_j^+|)f\|_{ L^{p_d,  2}(\mathbb R^d)+B}.
\end{equation}

To summarize, we only need to consider
\[
	 T_\lambda ^{+,\mu} f(x',x_d)\coloneqq \int_{\mathbb R^d}e^{ix\cdot \xi} \frac{\mu_{N, \gamma}( x_d)}{\mu_{N, \gamma}( y_d)} \frac{\chi_\lambda(\xi) \widehat f(\xi)}{P_m(\xi)-\lambda-i0}d\xi.\]
What we need to prove now is that $T_\lambda^{+,\mu} :  L^{p_d,  2}(\mathbb R^d)+B\to L^{p_d'}(\mathbb R^d)\cap B^*$ boundedly. Following the proof of  Proposition \ref{main proposition}, we can rewrite it as the following:
 \begin{align}
 T_\lambda ^{+,\mu} f(x',x_d)
  	&=\int_{\mathbb R}\int_{\mathbb R^{d-1}}\frac{\mu_{N, \gamma}( x_d)}{\mu_{N, \gamma}( y_d)} K_\lambda^+( x'-y', x_d-y_d) f(y',y_d)dy' dy_d \nonumber\\
  	&\  +\nonumber\int_{\mathbb R^d}\frac{\mu_{N, \gamma}( x_d)}{\mu_{N, \gamma}( y_d)} \ift{ G_{\lambda}}(x-y)f(y)dy \\
  	&=:  T_\lambda ^{+,\mu, \textrm{bad}} f(x',x_d)+ T_\lambda ^{+,\mu, \textrm{good}} f(x',x_d)
  \end{align}
where  $K_\lambda^+( x', x_d) $ is given by \eqref{first-expr-for-K+} and $ \ift{ G_{\lambda}} \in \mathcal S$.
For $T_\lambda ^{+,\mu, \textrm{good}} f(x',x_d)$, it is bounded from $ L^{p_d,  2}(\mathbb R^d)+B$ to $ L^{p_d'}(\mathbb R^d)\cap B^*$.

To deal with the $T_\lambda ^{+,\mu, \textrm{bad}} f(x',x_d)$ term, we denote
\[K_\lambda^{+, \mu}( x', x_d, y_d)=\frac{\mu_{N, \gamma}( x_d)}{\mu_{N, \gamma}( y_d)} K_\lambda^+( x', x_d-y_d).\]
Then, its partial Fourier transform with respect to $x'$ is
\begin{align}
   \ftp{K_\lambda^{\mu,+}} (\xi', x_d, \  y_d)= \frac{\mu_{N, \gamma}( x_d)}{\mu_{N, \gamma}( y_d)}i e^{ i x_d \varphi_\lambda (\xi')} H(x_d-y_d)Q_{\lambda}(\xi').
\end{align}

 Since $\mu_{N, \gamma}(x_d)\leq \mu_{N, \gamma}(y_d) $ when $H(x_d-y_d)$ does not vanish. In this case, $|\frac{\mu_{N, \gamma}( x_d)}{\mu_{N, \gamma}( y_d)}| \leq 1$. Therefore,
    \begin{align}
  	\sup_{x_d,y_d} \big \| \ftp{K_\lambda^{\mu,+}} (\xi', x_d, \  y_d)\big \|_{L^{\infty}(\mathbb R^{d-1})}<\infty.
  	\end{align}
By repeating the proof in \eqref{B to B^* bound}, we see that $T_\lambda ^{+,\mu}$ is $B\to B^*$ uniformly bounded in $\lambda$.

Similarly, defining $h_{x_d}^\mu (y', y_d)\coloneqq  \frac{\mu_{N, \gamma}( x_d)}{\mu_{N, \gamma}( y_d)}  i H(x_d-y_d)f(y', y_d)$ and following the proof of \eqref{From L^{p_d,2} to B^*}, we can prove $T_\lambda ^{+,\mu}$ is uniformly bounded $L^{p_d, 2}\to B^*$.

 Also,  $|K_\lambda^{\mu,+ }(x'-y', x_d, y_d)|\leq C_{\delta,d} (1+|x'-y'|+|x_d-y_d|)^{-\frac{d-1}2}$, thus we have that $T_\lambda ^{+,\mu}$ is bounded $L^{p_d, 2}\to L^{p_d',2}$.
   This finishes the proof of the Corollary.
\end{proof}


\section{Limiting absorption principle for high-order shr\"odinger operator with Admissible perturbations}\label{section-admis}

In this section, we prove that the same limiting absorption principle  of Section 3 is also valid for high-order Schr\"odinger operators with admissible perturbations. Most of the results and methods are standard in this section, thus we omit some details and refer the reader to \cite{I-S}. The main innovation is Lemma \ref{weighted estimates}, which is a consequence of the Rellich theorem in \cite{A-H}.

We first give the definition of an admissible perturbation.
\begin{definition}\label{admissible definition}
 We say that $V$ is an admissible perturbation of $P_m(D)$ if:
  \begin{enumerate}
  	\item \label{P1}$V\in\mathcal{B}(X^*,X)$ and
  $$\langle V\phi,\psi\rangle=\langle \phi, V\psi\rangle $$
  for any $\phi,\psi\in\mathcal{S}(\mathbb R^d)$.\\
  \item \label{P2}For any $\varepsilon>0$ and $N\geq0$, there exist constants $A_{N,\varepsilon},R_{N,\varepsilon}\in[1,\infty)$ such that
  \begin{align}\label{De-1}
   \|\mu_{N,\gamma}Vu\|_{X}\leq\varepsilon\|\mu_{N,\gamma}u\|_{X^*}+A_{N,\varepsilon}\|u1_{\{|x|\leq R_{N,\varepsilon}\}}\|_{L^2}
   \end{align}
  for any $u\in X^*$ and any $\gamma\in(0,1]$.
  \item \label{P3}there is an integer $J\geq 1$ and operators $A_j, B_j \in \mathcal B(X^*, L^2)$ for $1\leq j \leq J$ such that for any $f, g \in X^*$,
		\begin{align}
 \langle Vf, g \rangle=\sum_{j=1}^J \langle B_j f, A_j g\rangle.
		\end{align}
  \end{enumerate}
\end{definition}

\begin{remark}\label{symmetric and compactness}
 Notice that $\mathcal S$ is not dense in $X^*$, thus  property \eqref{P1}  of an admissible perturbation alone is not enough to guarantee that $V$ is a symmetric operator on $X^*$. However, by the property \eqref{P2} and the  argument in  \cite{I-S}, one can prove that $V$ is symmetric on $X^*$. In addition, the property \eqref{P2} also guarantee that $V$ is a compact operator from  $X^*$ to $X$. In particular, $R^m_0(i)V$ is a compact operator on $W_m^{2}$.
\end{remark}
 Some trivial modifications of Proposition 1.4 in \cite{I-S} gives some examples of admissible perturbations of $P_m(D)$; we will omit the details (one can simply replace $q_0=d/2$ by $q_m=\frac{d}{2m}$ or $q_m>1$, $d\leq 2m$).

 In this paper, we focus on the potentials in  
 the completion of $C_c^\infty(\mathbb R^d)$ with respect to the norm $\|\cdot\|_{L^{q,\infty}}$, which we denote by $L_0^{q,\infty}(\mathbb R^d)$.
\begin{prop} Let $V$ be a  real valued potential and
 \begin{align}\label{potential type 1}
    V&\in L_0^{q,\infty}(\mathbb R^d) \text{ for some }  q \in \big[q_m, \frac{d+1}{2}\big],
\end{align}
then $V$ is an admissible perturbation.
\end{prop}
\begin{proof} The symmetry follows since  $V$ is real-valued.

Following Lemma 6.1 in \cite{I-S}, it is easy to prove that for any $\ f\in L^{p_d',2}(\mathbb R^d), $
  \begin{align}
     \||V|^{1/2}S_{-\theta_{m,d}}f\|_{L^2(\mathbb R^d)}\leq C\|V\|_{L^{q_0,\infty}(\mathbb R^d)}^{1/2}\|f\|_{L^{p_d', 2}(\mathbb R^d)}.
  \end{align}
This implies that $V$ maps  from $X^*$ to $X$.

Finally, for any $V$ that satisfies  \eqref{potential type 1}, 
 we can find  $V_n \in C_c^\infty(\mathbb R^d)$ such that $V_n$ converge to $V$ in the
 ${L^{q,\infty}(\mathbb R^d)}$ norm. So condition \eqref{P3} can be easily deduced.
\end{proof}
Next, we construct a function which shows that $L^{q,\infty}_0$ is not contained in $L^q$.
\begin{example} Let $q\in[1,\infty)$ and let $\{E_j\}_{j=1}^\infty$ be a disjoint family of subsets of $\mathbb R^d$, each with measure $|E_j(x)|=\frac{1}{\ln(1+j)}$ and define
 $V(x)=\sum_{j=1}^\infty \frac{1}{j^{1/q}}{\mathbf{1}_{E_j}}(x).$ A direct computation shows that
$V(x)\in L^{q,\infty}(\mathbb R^d)$, but $V(x) \notin L^{q}(\mathbb R^d)$.

Now we prove that it can be approximated by a bump function in $L^{q, \infty}$ norm. Let us define $V_N(x)=\sum_{j=1}^{N-1} \frac{1}{j^{1/q}}{\mathbf{1}_{E_j}}(x)$.  In fact, for any given $\varepsilon>0$, we can choose $N$ large enough so that (writing $\lfloor t\rfloor$ for the integer part of $t\ge 0$,)
\begin{align*}
 \sup_{\lambda \in (0, \infty)} \lambda   \left|\left\{x\in\mathbb R^d\,\middle|\, |V(x)-V_N(x)|>\lambda \right\}\right|^{1/q}&\leq  \sup_{\lambda \in (0, N^{-1/q}]}\lambda \left(\sum_{j=N}^{\lfloor 1/\lambda^{q} \rfloor}\frac1{\ln(1+j)}\right)^{1/q}\\
 & \leq  \frac{\sup_{\lambda\in(0,N^{-1/q}]} \lambda \lfloor \frac1{\lambda^q}\rfloor ^{1/q}}{\ln(2+N)^{1/q}} <\frac{\varepsilon}{2}.
\end{align*}
Since $V_N \in L^{q}(\mathbb R^d)$ is compactly supported, there exists $\tilde V_N \in C_0^\infty $ such that $\|V_N-\tilde V_N\|_{L^{q}(\mathbb R^d)}<\frac{\varepsilon}{2}.$ It follows \begin{align*}
    \|V-\tilde V_N\|_{L^{{q}, \infty}(\mathbb R^d)}&\leq   \|V- V_N\|_{L^{{q}, \infty}(\mathbb R^d)}+  \|V_N-\tilde V_N\|_{L^{{q}, \infty}(\mathbb R^d)}\\
    &\leq \frac{\varepsilon}{2} +\frac{\varepsilon}{2} =\varepsilon.
\end{align*}
\end{example}

For the rest of this paper, we focus on proving Theorem \ref{Main theorem 2}.
\begin{proof}[Proof of Theorem \ref{Main theorem 2}]
For the first part of the theorem, it is standard that $H_m=P_m(D) +V$ is self-adjoint on $D(H_m)\coloneqq\{u\in W_m^2(\mathbb R^d): H_m u\in L^2(\mathbb R^d)\}$; this follows from the proof of Theorem 1.3 in \cite{I-S}.

Now, let us consider the limiting absorption principle for
 $H_m$. Since we already know that $H_m$ is  self-adjoint on $D(H_m)$, let us temporarily define three apparently distinct sets
 \begin{align*}
  \mathcal E &= \{ \text{non-zero eigenvalues of } H_m\},\\
  \widetilde  {\mathcal E}^{\pm}  &=\{\lambda \in \mathbb R\setminus \{0\}\ |\  \text{there exists}\  0\neq u\in X^* \  \text {such that } u+R^m_0 (\lambda\pm i0)V u=0\}.
\end{align*}
Notice that $R_0(\lambda+i0)\bar{f}=\overline{R_0(\lambda-i0)f}$ for any $f\in X$. Thus, we have $\widetilde  {\mathcal E}^-=\widetilde  {\mathcal E}^+=:\widetilde  {\mathcal E}$.
For any $\lambda\neq 0$, we denote the generalized eigenspaces by
\begin{align*}
  E_\lambda^{\pm}:=\{f\in X^*:(I+R_0(\lambda\pm i0)V)f=0\}.
\end{align*}
\begin{lemma}\label{Le-finite}
  Let $V$ be an admissible potential. Then:

  $(i)$ For any $\lambda\in \widetilde  {\mathcal E}$,
  $ E_\lambda^{\pm}$ are the eigenspaces of self-adjoint operator $H_m$.

  $(ii)$ The set $ \widetilde  {\mathcal E}$ is discrete in $\R\setminus\{0\}$, and for any $\lambda\in \widetilde  {\mathcal E}$, the vector spaces $ E_\lambda^{\pm}$ are finite dimensional.
\end{lemma}
\begin{proof}
  $(i)$ For any $u\in E_\lambda^{\pm}$, by the definition of $\widetilde {\mathcal E}$, we have
	\begin{align}\label{singular equation}
		u+R^m_0(\lambda\pm i0)Vu=0.
	\end{align}
 Now, we only need to show that $u\in D(H_m)$ and $H_m u=\lambda u$.	
 Note that $(P_m(D)-\lambda) R^m_0(\lambda\pm i0)g=g$ for any $g\in X$.
Using this fact and \eqref{singular equation},  we have
\begin{align*}
	(P_m(D)-\lambda)u+(P_m(D)-\lambda) R^m_0(\lambda\pm i0)Vu=0,
\end{align*}
which is equivalent  to
\begin{align} \label{differential equation }
	(P_m(D)+V)u=\lambda u.
\end{align}
	By Corollary \ref{Weighted uniform bound}, we have that
	\begin{align*}
	\|\mu_{N,\gamma} u\|_{X^*}&=\|\mu_{N,\gamma}R^m_0(\lambda\pm i0)Vu \|_{X^*}
	\\
	&\leq C_{N}\|\mu_{N,\gamma} Vu\|_{X}\\
	&\leq C_{N} \varepsilon \|\mu_{N,\gamma} u\|_{X^*}+C_{N} A_{N,\varepsilon}\|u{ \mathbf{ 1}}_{|x|\leq R}\|_{L^2}.
\end{align*}
We can take $\varepsilon\ll1 $ so that $C_{N} \varepsilon <1/2 $.  Then
\begin{align*}
	\|\mu_{N,\gamma} u\|_{X^*}\leq C_{N,\gamma}\|u\|_{B^*}.
\end{align*}
Taking $\gamma \to 0$ yields
\begin{align*}
	\|(1+|x|^2)^{N} u\|_{X^*}\leq C_{N,V, \gamma }\|u\|_{X^*}.
\end{align*}
This shows that $u $ has sufficiently fast decay.
In particular, for any $N\geq0$, we have $(1+|x|^2)^N u \in W_{m}^2$, and $u\in  W_m^2$. Using \eqref{differential equation } and the regularity of $u$, we have shown that $u\in D(H_m)$, which implies that $u$ is the eigenfunction corresponding to eigenvalue $\lambda$ of $H$.

 $(ii)$  This follows from the argument of Lemma 4.5 in \cite{I-S}, using the compactness of the operators $R_0^m(-1)(1+|x|^2)^{-1}$ and $R_0^m(-1)V$ on $W_m^2$.
\end{proof}

The above lemma implies that $\widetilde  {\mathcal E}\subset \mathcal E$. Next, we will prove that the reverse inclusion also holds.
 We first consider the equation
\begin{align}\label{differential equation}
    (P_m(D)-\lambda)u=f.
\end{align}
 The Rellich theorem established by S. Agmon and L. H\"ormander in  \cite{A-H} asserts that if $u \in B^*$, $f\in B$ and
\begin{align}\label{Rellich theorem condition}
	\lim_{R\to \infty}\frac{1}{R}\int_{|x|\leq R} |u|^2dx =0,
\end{align}
then there is a unique solution given by $u=R_0^m(\lambda +i0)f=R_0^m(\lambda -i0)f$.
The following lemma extends this result from $B^*$ to $X^*$.
\begin{lemma}\label{weighted estimates}
Let $u \in X^*$ and $\lambda \in \mathbb R\setminus \{0\}$ safisfy \eqref{differential equation} where $f\in X$. If we further assume that $u$ satisfies \eqref{Rellich theorem condition}, then we have
\begin{align}\label{solution representation}
  u=R_0^m(\lambda +i0)f=R_0^m(\lambda -i0)f
\end{align}
In particular, by Corollary \eqref{Weighted uniform bound}, we have
\begin{align}\label{Le6-5}
    \|\mu_{N,\gamma}u\|_{X^*}\leq C_N\|\mu_{N,\gamma}(P_m(D)-\lambda)u\|_{X}
  \end{align}

\end{lemma}

\begin{proof}
For any $f\in X$, a direct computation shows that
\begin{align}
  u_0=[R_0^m(\lambda+i0)+R_0^m(\lambda-i0)]f/2
\end{align}
is a special solution for \eqref{differential equation} and $u_0\in X^*\subset B^*$.
Thus, for any other solution of \eqref{differential equation} we have $(P_m(D)-\lambda)(u-u_0)=0$.  By Theorem 4.1 in \cite{A-H}, there exists an  $L^2$-density $v_0$ on surface $\mathbb{S}^{d-1}_{r(\lambda)}$ , such that  $u-u_0 =\mathcal F^{-1}(v_0 d\sigma_\lambda)$, where $d\sigma_\lambda$ is the surface measure for $\mathbb{S}^{d-1}_{r (\lambda)}$. That is,
\begin{align}
    u=[R_0^m(\lambda+i0)+R_0^m(\lambda-i0)]f/2+\mathcal F^{-1}(v_0 d\sigma_\lambda).
\end{align}
For $f\in \mathcal{S}$, by Theorem 6.1 in \cite{A-H},
\begin{align}\label{special solution}
    \int_{\mathbb{S}^{d-1}_{r(\lambda)}} (|v_0|^2+ |\pi \hat f|^2 |P_m'|^{-2}) d\sigma_\lambda \leq \pi \liminf_{R\to \infty}
    \frac{1}{R}\int_{|x|<R}  |u|^2 dx.
\end{align}
In view of Proposition \ref{main proposition} and the sharp trace lemma,
both sides of \eqref{special solution}
 are continuous  with respect to the $X$ norm of $f$  and the $L^2(d\sigma_\lambda)$ norm of $v_0$. As $\mathcal S $ is dense in $X$, it follows that the inequality \eqref{special solution} is valid for all $f\in X$.
 By assumption,  the right hand side is $0$, thus $v_0=0$ and $\hat f=0$ on $\mathbb S_{r(\lambda)}^{d-1}$. This shows that $u=R^m_0(\lambda+i0)f=R^m_0(\lambda-i0)f$.
\end{proof}

\begin{prop} $\mathcal E = \widetilde {\mathcal E} $.
\end{prop}

\begin{proof} By Lemma \ref{Le-finite}, we only need to prove that $\mathcal E \subset \widetilde {\mathcal E}$. Take any $\lambda \in \mathcal E $, and $0 \neq u \in D(H_m) \subset X^*$ such that $(P_m(D)-\lambda) u=-Vu: =f \in X$. In particular, $u\in L^2(\mathbb R^d)$, which implies that
\begin{align}
	\lim_{R\to \infty}\frac{1}{R}\int_{|x|\leq R} |u|^2dx =0.
\end{align}
 By Lemma \ref{weighted estimates},  we have $u=R_0^m(\lambda \pm i0)f=-R_0^m(\lambda \pm i0) Vu$, thus $\lambda \in \tilde {\mathcal E} $.
Hence, $\widetilde {\mathcal E}\subset \mathcal E,$ which proves the proposition.
\end{proof}

Above results imply that the non-zero eigenvalues $\mathcal{E}$ is discrete in $\R\setminus\{0\}$ and each eigenvalue has finite multiplicity.

Next, we prove the second part of Theorem \ref{Main theorem 2}.
 The following lemma gives locally uniform bounds for the inverse operators, which we use to establish the limiting absorption principle for the perturbed operator.

\begin{lemma}\label{uniform bounded in X^*}
For any compact set $I$ contained in
$\mathbb R\setminus (\{0\}\cup {\mathcal E})$,
  \begin{align}\label{uniform lemma}
    \sup_{\lambda\in I, 0\leq \varepsilon\leq 1}\left\|({\rm{Id}}_{X^*}+R^m_0(\lambda\pm i \varepsilon )V)^{-1}\right\|_{X^*\rightarrow X^*}\leq C(V,I).
  \end{align}
\end{lemma}

The proof for Lemma  \ref{uniform bounded in X^*} is standard, the interested reader can check \cite{I-S} for more details.

\label{subsection-proof-main-theorem-2}
For $\lambda \in I$, using fact the resolvent identity
\begin{align}
  R^m(\lambda\pm i\varepsilon )=({\rm{Id}}_{X^*}+R^m_0(\lambda\pm i \varepsilon )V)^{-1}R_0^m(\lambda\pm i\varepsilon ),
\end{align}
together with Theorem \ref{Main theorem -1} and  Lemma \ref{uniform bounded in X^*}, we can get the estimate \eqref{LAP-V}.
From \eqref{LAP-V} and Theorem XIII.20 in \cite{Reed-Simon4}, we see that the singular continuous spectrum is null, i.e. $\sigma_{sc}(H_m)=\varnothing$.

Next, we prove that $\sigma_{ac}(H_m)\subset[0,\infty)$. Assume that $\lambda \in (-\infty, 0)\setminus \mathcal{E}$, we can show that $\lambda$ belongs to the resolvent set $\rho(H_m)$ of $H_m$ as follows.
Since $\lambda $ is not an eigenvalue of $H_m$, the equation $f+R^m_0(\lambda\pm i\varepsilon  ) Vf=0 $ has no solution in $W_m^2$.
Hence, by the Fredholm alternative, ${\rm{Id}}_{W_m^2}+R_0^m(\lambda) V$ is invertible on $W_m^2$.
But now, just notice that $[{\rm{Id}}_{W_m^2}+R_0^m(\lambda) V]^{-1}R_0^m(\lambda)(H_m -\lambda)$ is the identity map, and so is $(H_m -\lambda)[{\rm{Id}}_{W_m^2}+R_0^m(\lambda) V]^{-1}$, which is to say that $\lambda$ belongs to the resolvent set of $H_m$ as claimed.

In addition, using condition \eqref{P3} of an admissible perturbation and the limiting absorption principle for both $P_m(D)$ and $H_m$,
it is routine to check that the operators $A_j, B_j$ are both $H_m$-smooth and $P_m(D)$-smooth on compact subsets  $I\subset \mathbb R\setminus (\{0\}\cup \mathcal E)$. Hence, the local wave operators $\Omega^{\pm}$ exist and are complete. By Lemma \ref{Le-finite}, $(0, \infty)\setminus \mathcal E= \cup_{i=1}^\infty (a_i, b_i)$. That the local wave operators exist and are complete implies in turn that $(a_i, b_i)\subset \sigma_{ac}(H_m)$. Therefore, we in fact have $[0,\infty)\subset \overline{\bigcup_{i=1}^\infty (a_i, b_i)}\subset \sigma_{ac}(H_m)$. In conclusion, we have proven that $\sigma_{ac}(H_m)=[0, \infty)$, which finishes the proof of Theorem \ref{Main theorem 2}. 
\end{proof}

\end{document}